\documentclass[12pt,reqno]{amsart}

\usepackage{amssymb,amsmath,amsthm}
\usepackage{color}
\usepackage{cite}

\newtheorem{theorem}{Theorem}[section]
\newtheorem{lemma}[theorem]{Lemma}
\newtheorem{prop}[theorem]{Proposition}
\newtheorem{cor}[theorem]{Corollary}
\newtheorem{remark}[theorem]{Remark}

\def \R{\mathbb{R}}

\def \Rn{\mathbb{R}^n}

\def \e{\varepsilon}

\def \grad{\nabla}

\def \A{\mathcal{A}}

\def \half{\frac{1}{2}}

\numberwithin{equation}{section}

\begin{document}

\title[Acousto-Optic Tomography]{Inverse Transport and Acousto-Optic Imaging}

\author[Chung]{Francis J. Chung}
\address{Department of Mathematics, University of Kentucky, Lexington, KY, USA}
\email{fj.chung@uky.edu}

\author[Schotland]{John C. Schotland}
\address{Department of Mathematics and Department of Physics, University of Michigan, Ann Arbor, MI, USA}
\email{schotland@umich.edu}

\date{\today}

\begin{abstract}
We consider the inverse problem of recovering the optical properties of a highly-scattering medium from acousto-optic measurements. Using such measurements, we show that the scattering and absorption coefficients of the radiative transport equation can be reconstructed with Lipschitz stability by means of algebraic inversion formulas. 
\end{abstract}



\maketitle

\section{Introduction}
\subsection{Background}

The development of effective methods for optical imaging of highly-scattering media is a problem of considerable practical importance~\cite{Arridge_2009}. We note that biomedical applications are of particular interest, since optical methods are widely employed to image physiological function and various biomolecular processes. In an optical imaging experiment, a medium of interest is illuminated by a narrow collimated beam and the light that propagates through the medium is collected by an array of detectors. The optical properties of the medium are then reconstructed by solving an inverse problem, a typical example being to recover the coefficients of an elliptic partial differential equation from boundary measurements. It is well known that such problems are severely ill-posed, which leads to reconstructed images with relatively low spatial resolution~\cite{Uhlmann_2009,Bal_2009}.

Acousto-optic tomography (AOT) is a recently proposed method that mitigates certain limitations of optical imaging. The physical principle is to perform an optical imaging experiment in which the optical properties of the medium are spatially modulated by an acoustic wave. The associated inverse problem consists of two steps. In the first step, by proper choice of the acoustic field together with boundary measurements of the optical field, a functional of the unknown coefficients is recovered. This functional, which serves as a proxy for measurements of the optical field, is known everywhere in the medium. The second step consists of recovering the unknown coefficients from the internal functional. This inverse problem is well-posed, resulting in reconstructions with good spatial resolution. See~\cite{ammari_2008,bal_2012_a,bal_2010,bal_2012_b,bal_2012_c,bal_2013,BalSchotland_PhysRev2014,BCS,capdeboscq_2009,gebauer_2009,kuchment,kuchment_2012,monard_2011,mclaughlin_2004,mclaughlin_2010,nachman} for examples of multi-wave inverse problems in other physical settings.

The standard approach to modeling the propagation of light in AOT makes use of the diffusion approximation (DA) to the radiative transport equation (RTE)~\cite{BalSchotland,BalMoskow,Ammari_2014_1,Ammari_2014_2,Ammari_2013}. The DA breaks down in optically thin layers, in weakly scattering or strongly absorbing media, and near boundaries. One or more of these conditions is often met in biomedical applications. In this paper, we consider the inverse problem of AOT within the framework of radiative transport theory. We find that the attenuation and scattering coefficients of the RTE can be reconstructed with Lipschitz stability by means of \emph{algebraic} inversion formulas. In contrast, we note that for the case of AOT within the DA, only iterative reconstruction methods have been proposed~\cite{BalSchotland,BalMoskow,Ammari_2014_1,Ammari_2014_2,Ammari_2013}.
 
\subsection{Main Results}

Let $X$ be a bounded domain in ${\mathbb R}^n$ with smooth boundary $\partial X$, for dimension $n\ge 2$. The propagation of multiply-scattered light is taken to be governed by the RTE 
\begin{align}
\label{RTE}
\theta \cdot \grad u + \sigma(x) u = \int_{S^{n-1}} k(x,\theta,\theta ')u(x,\theta ') d \theta'  \ .
\end{align}
Here $u(x,\theta)$ is the intensity of light at the point $x\in X$ 
traveling in the direction $\theta\in S^{n-1}$. The coefficients $\sigma$ and $k$ describe the attenuation and scattering, respectively, of light in $X$. We will assume that $\sigma$ belongs to $L^{\infty}(X)$ and $k$ is continuous. We will also assume that $k$ obeys the reciprocity relation 
\begin{equation}
\label{IsotropiK}
k(x,\theta, \theta ') = k(x,-\theta ', -\theta).
\end{equation}

To guarantee solvability of the RTE, we follow ~\cite{ChoulliStefanov} and assume that one of the following conditions holds: either an \emph{absorption condition}
\begin{equation}
\label{AbsorptionCondition}
\sigma - \rho \geq \alpha,
\end{equation}
for some positive constant $\alpha$, or a \emph{smallness condition}
\begin{equation}
\label{SmallnessCondition}
\tau \rho < 1.
\end{equation}
Here $\rho$ is defined by
\begin{equation}
\rho = \left\| \int_{S^{n-1}} | k(x,\theta, \theta')| d\theta'  \right\|_{L^{\infty}(X \times {S^{n-1}})}.
\end{equation}
We also define the subsets $\Gamma_{\pm}$ of $\partial X \times {S^{n-1}}$ by
\begin{equation}
\Gamma_{\pm} = \{(x,\theta) \in \partial X \times {S^{n-1}}| \pm \theta \cdot n(x)  > 0\}
\end{equation}
where $n(x)$ is the outward unit normal vector at $x$.  
Then we have the following existence result, which we state here in a form adapted from Theorem 2.1 in~\cite{BCS}.

\begin{prop}
\label{RegularityTheorem}
Let $f_{-} \in L^{\infty}(\Gamma_{-})$.  Under the conditions on $X$, $\sigma$, and $k$ given above, the equation \eqref{RTE} has a unique solution $u$ obeying the boundary condition
\[
u|_{\Gamma_{-}} = f_{-}.
\]
Moreover, for $1 \leq p \leq \infty$,
\begin{equation}
\label{SolutionEstimate}
\|u\|_{L^p({S^{n-1}}, L^{\infty}(X))} \leq C_p \|f_{-}\|_{L^p({S^{n-1}}, L^{\infty}(\partial X))}
\end{equation}
for some constant $C_p$ depending on $X$, $\sigma$, and $k$.
\end{prop}

For notational convenience, we define 
\begin{equation}
\label{DefnA}
Au = -\sigma u + \int_{S^{n-1}}k(x,\theta,\theta')u(x, \theta')d\theta'. 
\end{equation}
Then if we do not need to consider $\sigma$ and $k$ separately, we can write the RTE \eqref{RTE} in the form
\begin{equation}
\label{ShortRTE}
(\theta \cdot \grad - A)u = 0.
\end{equation}
Note that a result similar to Theorem~\ref{RegularityTheorem} is true if the RTE \eqref{RTE} is replaced by the adjoint equation
\begin{equation}\label{AdjointRTE}
(-\theta \cdot \grad - A)v = 0.
\end{equation}
In fact, if $u$ solves \eqref{RTE}, then a calculation shows that the function $\tilde{u}$ defined by $\tilde{u}(x,\theta) = u(x,-\theta)$ solves \eqref{AdjointRTE}.

The above existence result, combined with a trace theorem for the solutions to the RTE (see~\cite{DautrayLions}), means that we can define the albedo operator $\A: L^{\infty}(\Gamma_{-}) \rightarrow L^{\infty}(\Gamma_{+})$ by
\[
\A(f) = u|_{\Gamma_{+}}.
\]
The problem of recovering $\sigma$ and $k$ from $\A$ has been addressed by Choulli and Stefanov ~\cite{ChoulliStefanov} and reviewed in~\cite{Bal_2009}. The inverse problem of AOT is formulated as follows. Suppose that an acoustic pressure wave of the form $\cos(q \cdot x + \varphi)$ is incident on the medium, where $q$ is the wave vector  
and $\varphi$ is the phase of the wave. Following~\cite{BalSchotland}, we find that the coefficients $\sigma$ and $k$ are modulated according to 
\begin{eqnarray*}
\sigma_{\e} &=& (1 + \e \cos(q \cdot x + t)) \sigma \\
k_{\e} &=& (1 + \e \cos(q \cdot x + t)) k \ ,
\end{eqnarray*}
where $0 < \e \ll 1$ is the dimensionless amplitude of the acoustic wave.  The RTE (\ref{RTE}) thus becomes
\begin{equation}
\label{ModulatedRTE}
\theta \cdot \grad u_\e + \sigma_{\e} u_\e = \int_{{S^{n-1}}} k_{\e}(x,\theta,\theta ')u_\e(x,\theta ') \, d \theta ' ,
\end{equation}
where the dependence of $u$ on $\epsilon$ has been made explicit.
For sufficiently small $\e$, the conditions (\ref{AbsorptionCondition}) and (\ref{SmallnessCondition}) on $\sigma$ and $k$ ensure that (\ref{ModulatedRTE}) also has a unique solution.  Therefore, for suitable values of $q$ and $\varphi$, we can obtain new albedo maps $\A_{\e}(q,\varphi)$ defined by
\[
\A_{\e}(q,\varphi)(f) \mapsto u_\e|_{\Gamma_{+}},
\]
where $u_\e$ solves \eqref{ModulatedRTE} with the appropriate values of $\e, q,$ and $t$.  It will be convenient, for a fixed value of $f$, to view $\A_{\e}(q,\varphi)(f)$ as a function of $q$ and $\varphi$.  Then $\A_{\e}(f)$ is a map from $\R^n \times \R$ to $C(\Gamma_{+})$.  

The purpose of this paper is to show that the maps $\A_{\e}$ can be used to determine $\sigma$ and $k$.  More specifically, we have the following three results.  First, the maps $\A_{\e}$ can be used to recover an internal functional of $\sigma$ and $k$. 

\begin{prop}\label{InternalFunctional}
Suppose $f,g \in L^{\infty}(\Gamma_{-})$.  Let $u$ be the solution to the RTE \eqref{RTE} with boundary condition $u|_{\Gamma_-} = f$, and $v$ be the solution to the adjoint RTE \eqref{AdjointRTE} with boundary condition $v|_{\Gamma_+} = \tilde{g}$.  Then $\A(g)$ and $\A_{\e}(f)$ determine the internal functional $H \in L^{\infty}(X)$ defined by
\begin{equation}\label{Functional}
H(x) = \int_{S^{n-1}} Au \, v \, d\theta
\end{equation}
up to order $\e$. Moreover if $H_1$ and $H_2$ are functionals obtained from the same initial data $(f,g)$, but separate sets of coefficients $\sigma_1, k_1$ and $\sigma_2, k_2$, we have the stability estimate
\begin{eqnarray*}
\|H_1 - H_2\|_{L^{\infty}(X)} &\lesssim& \|g\|_{L^1(\Gamma_{-})}\|\A^1_{\e}(f)-\A^2_{\e}(f)\|_{L^1(\Rn \times \{0,\frac{\pi}{2}\}, L^{\infty}(\Gamma_{+}))} \\
& &  + \|f\|_{L^1(\Gamma_{-})}\|\A^1(g)-\A^2(g)\|_{L^{\infty}(\Gamma_{+})} + O(\e).\\
\end{eqnarray*}
\end{prop}

Since $H$ depends on the choices of the boundary conditions $f$ and $g$ for $u$ and $v$, we will sometimes write $H(f,g)(x)$ whenever we want to emphasize this distinction.  

The second and third results state that for appropriate choices of $f$ and $g$, the functionals $H(f,g)$ can be used to determine $\sigma$ and $k$. In the following two theorems, the lengths $\tau_{\pm} = \tau_{\pm}(x,\theta)$ are defined to be the distances from $x$ to $\Gamma_{\pm}$ in the direction of $\pm \theta$. In other words $\tau_{\pm}$ are defined so that $x \pm \tau_{\pm}\theta \in \Gamma_{\pm}$.

\begin{theorem}\label{AbsorptionRecovery}
Let $h > 0$ be small, and let $\theta_0$ be any fixed element of $S^{n-1}$.  There exists an $f_h \in L^{\infty}(\Gamma_{-})$, which is a function of the angular variable only, such that $H(f_h,f_h)(x)$ and $\A_0(h^{\half(n-1)}f_h)(x + \tau_{+}\theta_0,\theta_0)$ are $O(1)$, and
\begin{equation}\label{SigmaFromH}
\sigma(x) = \frac{H(f_h,f_h)(x)}{\A_0(h^{\half(n-1)}f_h)(x + \tau_{+}\theta_0,\theta_0)} + o(h).
\end{equation} 
\end{theorem}

\begin{theorem}\label{ScatteringRecovery}
Let $h > 0$ and suppose that $\sigma$ is known. There exists a family of functions $g^{\theta}_h \in L^{\infty}(\R^n \times S^{n-1})$, parametrized by $\theta \in S^{n-1}$, such that for $\theta_1 \neq \theta_2$, 
\begin{equation}
\label{KFromH}
k(x,\theta_2, \theta_1) = \left| H(g^{\theta_1}_h, g^{\theta_2}_h)\exp\left(\int_0^{\tau_{+}(x,\theta_2)}\sigma(x + s\theta_2)ds+\int_0^{\tau_{-}(x,\theta_1)}\sigma(x - s\theta_1)ds\right)\right| +o(h).
\end{equation} 
\end{theorem}

\begin{remark}
Given a scattering kernel $k(x, \theta, \theta')$ which depends only on $x$ and the angle between $\theta$ and $\theta'$, then we can make do with a one-parameter family of measurements by fixing $\theta_1$ and taking a one parameter family of $\theta_2$ so as to produce all angles between $\theta_1$ and $\theta_2$.  
\end{remark}

The formulas \eqref{SigmaFromH} and \eqref{KFromH} immediately imply the following Lipschitz stability estimates. 

\begin{theorem}
If $H_1$ and $H_2$ are functionals obtained from separate sets of coefficients $\sigma_1, k_1$ and $\sigma_2, k_2$, then
\[
\|\sigma_1 - \sigma_2\|_{L^{\infty}(X)} \lesssim \|H_1(f_h,f_h) - H_2(f_h,f_h)\|_{L^{\infty}(X)} + o(h)
\]
and
\[
\|k_1 - k_2\|_{L^{\infty}(X\times {S^{n-1}}\times {S^{n-1}})} \lesssim \sup_{\theta_1,\theta_2}\left\| H_1(g^{\theta_1}_h,g^{\theta_2}_h) - H_2(g^{\theta_1}_h,g^{\theta_2}_h) \right\|_{L^{\infty}(X)} + o(h).
\]
\end{theorem}

The remainder of this paper is organized as follows. In Section 2, we will prove Proposition \ref{InternalFunctional}, in Section 3, we will prove Theorem \ref{AbsorptionRecovery}, and in Section 4, we will prove Theorem \ref{ScatteringRecovery}.

\section{Internal Functional}

In this section we will prove Proposition \ref{InternalFunctional}.  We begin by introducing the operator 
$A_{\e}$, which is defined by
\[
A_{\e} u = -\sigma_{\e} u + \int_{S^{n-1}}k_{\e}(x,\theta,\theta')u(x, \theta')d\theta'. 
\]
Then, the modulated RTE \eqref{ModulatedRTE} becomes
\[
(\theta \cdot \grad - A_{\e})u = 0.
\]
We also have that $A_{\e} - A$ is given by
\[
(A_{\e} - A) u = \e \cos(q \cdot x + \varphi) Au.
\]

We  note that the adjoint RTE defines a map $\tilde{\A}$ from $L^{\infty}(\Gamma_{+})$ to $L^{\infty}(\Gamma_{-})$, by analogy to the definition of $\A$ for the regular RTE.  In fact, $\A$ determines $\tilde{\A}$, since we have the relation
\[
\widetilde{\A(f)} = \tilde{\A}(\tilde{f}),
\]
where, for a given function $g$, the expression $\tilde{g}$ indicates the reflection of $g$ in the $\theta$ variable:
\[
\tilde{g}(x,\theta) = g(x,-\theta).
\]

\begin{proof}[Proof of Proposition \ref{InternalFunctional}]

Suppose $u_{\e}$ solves the modulated RTE \eqref{ModulatedRTE} with boundary condition $u_{\e}|_{\Gamma_{-}} = f$ and $v$ solves \eqref{AdjointRTE} with the boundary condition $v|_{\Gamma_{+}} = \tilde{g}$. Then $u_{\e}|_{\partial X}$ and $v|_{\partial X}$ are determined by $f,g, \A_{\e}(f)$ and $\A(g)$, so these boundary values are known. Next, we consider the expression
\[
\int_{X} \theta \cdot \grad u_{\e} v \, dx.  
\]
Integrating by parts, we obtain
\[
\int_{X} \theta \cdot \grad u_{\e} v \, dx = - \int_{X} u_{\e} \theta \cdot \grad v\, dx + \int_{\partial X} u_{\e} v \, n \cdot \theta \, dx,
\]
where $n$ is the outward unit normal vector on $\partial X$. We can make substitutions for $\theta \cdot \grad u_{\e}$ and $\theta \cdot \grad v$ using the equations \eqref{ModulatedRTE} and \eqref{AdjointRTE} respectively, to get
\[
\int_{X} A_{\e} u_{\e} v \, dx = \int_{X} u_{\e} A v\, dx + \int_{\partial X} u_{\e} v \, n \cdot \theta \, dx. 
\]
Now if we integrate in the $\theta$ variables, we find 
\[
\int_{X \times {S^{n-1}}} A_{\e} u_{\e} v \, dx \, d\theta = \int_{X \times {S^{n-1}}} u_{\e} A v\, dx\, d\theta + \int_{\partial X \times {S^{n-1}}} u_{\e} v \, n \cdot \theta \, dx\, d\theta. 
\]
In this setting the operators $A_{\e}$ and $A$ are self-adjoint, and thus
\begin{equation}\label{PreFunctional}
\int_{X \times {S^{n-1}}} (A_{\e} - A) u_{\e} v \, dx \, d\theta = \int_{\partial X \times {S^{n-1}}} u_{\e} v \, n \cdot \theta \, dx\, d\theta. 
\end{equation}
The right hand side of the above equation is known, since the boundary values of $u_{\e}$ and $v$ are known.  Therefore the left side of \eqref{PreFunctional} is also known. As noted in ~\cite{BCS}, $u_{\e} = u + O(\e)$, where $u$ is the solution to the unmodulated RTE \eqref{RTE} with the same boundary values as $u_{\e}$.  Therefore the left side of \eqref{PreFunctional} becomes
\[
\int_{X \times {S^{n-1}}} \e\cos(q \cdot x + \varphi) A u v \, dx \, d\theta + O(\e^2).
\]
Therefore to first order in $\e$, we can recover the quantity 
\[
\int_{X \times {S^{n-1}}} \e\cos(q \cdot x + \varphi) A u v \, dx \, d\theta.
\]
By varying $q$ and $\varphi$, we obtain the Fourier transform of the function $H(x)$ defined by 
\[
H(x) = \int_{{S^{n-1}}} A u v \, d\theta.
\]
If we take two different sets of coefficients, forming the operators $A^1$ and $A^2$, and examine the resulting functionals $H_1$ and $H_2$, then the above reasoning tells us that
\[
{H}_1 - {H}_2 = \int_{\partial X \times {S^{n-1}}} (u^1_{\e} v^1 - u^2_{\e} v^2) \, n \cdot \theta \, dx\, d\theta + O(\e)
\]
The stability estimate then follows by applying the estimates from Theorem \ref{RegularityTheorem} on the right side.  
\end{proof}

\section{Recovering the Absorption Coefficient}

In this section we prove Theorem \ref{AbsorptionRecovery}. To begin, we indicate
the relationship between $u$ and its boundary value $f$ on $\Gamma_{-}$, which follows from a result in~\cite{ChoulliStefanov}.  To state this result, we will define the following operators, using notation from~\cite{ChoulliStefanov}.  Let $\tau_{\pm}(x,\theta)$ be the distance from $x$ to $\Gamma_{\pm}$ in the $\theta$ direction. We define $J$ to be the operator
\[
Jf(x,\theta) = \exp\left(-\int_0^{\tau_{-}(x,\theta)}\sigma(x - s\theta)ds\right)f(x - \tau_{-}(x,\theta)\theta, \theta),
\]
and $T_1^{-1}$ to be the operator
\[
T_1^{-1}w = \int_0^{\tau_{-}(x,\theta)}\exp\left(-\int_0^{t}\sigma(x - s\theta)ds\right)w(x - t\theta, \theta)dt.
\]
Finally, we define
\[
A_2 w = \int_{S^{n-1}} k(x, \theta, \theta') w(x,\theta ') d \theta ',
\]
and 
\[
Kw = T^{-1}_1 A_2 w. 
\]
The following result is essentially from~\cite{ChoulliStefanov}, with the exception of the $L^{\infty}$ estimate.

\begin{prop}\label{UDecomp}
Suppose $f \in L^{\infty}(\Gamma_{-})$, and $u$ solves the RTE \eqref{RTE} with the boundary condition $u|_{\Gamma_{-}} = f$.  Then $u$ takes the form
\[
u = Jf + \sum_{j = 1}^{\infty} K^j (Jf),
\]
Moreover, $\|K\|_{L^{\infty}(X \times {S^{n-1}}) \rightarrow L^{\infty}(X \times {S^{n-1}})} < c$ for some constant $c < 1$.  
\end{prop}

As a consequence, we have the following corollary, which tells us that if the $L^1$ norm of $f$ is small, then the solution $u$ is essentially just $Jf$ up to a higher order error.
\begin{cor}\label{PreservingSingularity}
Let $f \in L^{\infty}({S^{n-1}})$ with $\|f\|_{L^1({S^{n-1}})} = h$. Suppose $u$ solves \eqref{RTE} with boundary condition $u(x,\theta) = f(\theta)$ on $\Gamma_{-}$.  Then for small $h$ we have
\[
\|u  - Jf\|_{L^{\infty}(X\times S^{n-1})} = O(h).
\]  
\end{cor}

\begin{proof}
Note that at any $x \in X$,
\begin{equation}\label{JfL1}
\|Jf(x, \cdot)\|_{L^{1}({S^{n-1}})} \leq \|f\|_{L^{1}({S^{n-1}})},
\end{equation}
and
\begin{equation}\label{A2wLinfty}
\|A_2(w)(x, \cdot)\|_{L^{\infty}({S^{n-1}})} < \rho\|w(x, \cdot)\|_{L^1({S^{n-1}})}.
\end{equation}
Moreover
\[
\|T_1^{-1}w\|_{L^{\infty}(X \times {S^{n-1}})} < \|w\|_{L^{\infty}(X \times {S^{n-1}})}.  
\]
Thus
\begin{eqnarray*}
\|KJf \|_{L^{\infty}({S^{n-1}})} &\leq& \|A_2 Jf\|_{L^{\infty}({S^{n-1}})} \\ 
                         &\leq& \rho\|Jf\|_{L^1({S^{n-1}})} \\
                         &\lesssim& \|f\|_{L^1({S^{n-1}})} \\
												 &=& O(h) \\
\end{eqnarray*}
Since $\|K\|_{L^{\infty}(X \times {S^{n-1}}) \rightarrow L^{\infty}(X \times {S^{n-1}})} < c$ for some constant $c < 1$, all of the terms in $u- Jf$ are $O(h)$, and so the result now follows from the previous proposition.
\end{proof}

Note that a version of Corollary \ref{PreservingSingularity} also holds for the adjoint solution $v$.  If we define the operator $\tilde{J}$ by
\[
\tilde{J}f(x,\theta) = \exp\left(-\int_0^{\tau_{+}(x,\theta)}\sigma(x + s\theta)ds\right)f(x + \tau_{+}(x,\theta)\theta, \theta),
\]
and take $v(x,\theta) = f(\theta)$ on $\Gamma_+$, with $f$ as in the statement of Corollary \ref{PreservingSingularity}, then we have
\[
\|v  - \tilde{J}f\|_{L^{\infty}(X\times S^{n-1})} = O(h).
\]
Because $f$ is defined as a function from ${S^{n-1}}$ to $\R$, both the boundary conditions $u(x,\theta)|_{\Gamma_{-}} = f(\theta)$ and $v(x,\theta)|_{\Gamma_{+}} = f(\theta)$ are well defined.

The main idea behind the proof of Theorem \ref{AbsorptionRecovery} is to fix a direction $\theta_0$ and let $f_h \in L^{\infty}(S^{n-1})$ be functions that approximate the square root of the delta function $\delta(\theta - \theta_0)$ as $h \rightarrow 0$.  Then we can check that $\|f\|_{L^1(S^{n-1})}$ is $O(h)$, and use the facts that $u = Jf_h + O(h)$ and $v = \tilde{J}f_h + O(h)$ to rewrite the functional $H(f_h,f_h)$ in terms of $Jf_h$ and $\tilde{J}f_h$ up to an error of size $O(h)$. To make this more precise, we require the following lemma.

\begin{lemma}
\label{SingularFunctional}
Let $h > 0$ and let $f_h: L^{\infty}({S^{n-1}})$ be defined by 
\[
f_h(\theta) = \left\{ \begin{array}{ll} h^{\half(1-n)} & \mbox{ \rm if } \ |\theta - \theta_0| < h \\ 
                                        0              & \mbox{\rm \ otherwise }  \end{array} \right.
\]
for some $\theta_0 \in {S^{n-1}}$. Then
\[
H(f_h,f_h) = \int_{{S^{n-1}}} \sigma Jf_h \tilde{J}f_h \, d\theta + O(h^2).
\]
\end{lemma}

\begin{proof}
Let $u$ be the solution of \eqref{RTE} with the boundary condition $u|_{\Gamma{-}} = f_h$, and let $v$ be the solution of \eqref{AdjointRTE} with the boundary condition $v|_{\Gamma{+}} = f_h$. Then
\begin{eqnarray*}
H(f_h,f_h) &=& \int_{S^{n-1}} Au v  d\theta \\
           &=& \int_{{S^{n-1}}} \sigma uv \, d\theta + \int_{S^{n-1}} A_2(u) v d\theta. \\
\end{eqnarray*}
Writing $u = Jf_h + (u - Jf_h)$ and $v = \tilde{J}f_h + (v - \tilde{J}f_h)$ we can expand this to get
\begin{eqnarray*}
H(f_h,f_h) &=& \int_{{S^{n-1}}} \sigma Jf_h \tilde{J}f_h \, d\theta + \int_{S^{n-1}} \sigma Jf_h (v - \tilde{J}f_h) d\theta \\
	         & & + \int_{S^{n-1}} \sigma (u - Jf_h) \tilde{J}f_h d\theta+ \int_{S^{n-1}} \sigma (u - Jf_h) (v - \tilde{J} f_h) d\theta \\
	         & & + \int_{{S^{n-1}}} A_2(Jf_h) \tilde{J}f_h \, d\theta + \int_{S^{n-1}} A_2(Jf_h) (v - \tilde{J}f_h) d\theta \\
	         & & + \int_{S^{n-1}} A_2(u - Jf_h) \tilde{J} f_h d\theta+ \int_{S^{n-1}} A_2(u - Jf_h) (v - \tilde{J} f_h) d\theta. \\
\end{eqnarray*}

Now $\|f_h\|_{L^1({S^{n-1}})} = O(h)$, so we can use Corollary \ref{PreservingSingularity} and similar reasoning to show that all of the above terms except the first one are of higher order in $h$.  For example, the remark following Corollary \ref{PreservingSingularity} says that 
\[
\|v - \tilde{J}f_h\|_{L^\infty(X \times {S^{n-1}})} = O(h),
\]
and we know from \eqref{JfL1} that
\[
\|Jf_h(x, \cdot)\|_{L^1({S^{n-1}})} \leq \|f_h\|_{L^{1}({S^{n-1}})} = O(h)
\]
at any $x \in X$.  
Therefore the term 
\[
\int_{S^{n-1}} \sigma Jf_h (v - \tilde{J}f_h) d\theta
\]
is $O(h^2)$.  Similarly, we can use Corollary \ref{PreservingSingularity} and \eqref{JfL1} to show the terms
\[
\int_{S^{n-1}} \sigma (u-Jf_h) \tilde{J}f d\theta+ \int_{S^{n-1}} \sigma (u-Jf_h) (v-\tilde{J}f_h) d\theta
\]
are $O(h^2)$. Meanwhile, \eqref{A2wLinfty} implies that at any $x \in X$,
\[
\|A_2(Jf_h)(x, \cdot)\|_{L^{\infty}} \lesssim \|Jf_h(x, \cdot)\|_{L^1({S^{n-1}})} = O(h).
\]
Combining this with the fact that $\|\tilde{J}f_h(x, \cdot)\|_{L^1({S^{n-1}})} = O(h)$, we see that
\[
\int_{{S^{n-1}}} A_2(Jf_h) \tilde{J}f_h \, d\theta
\]
is $O(h^2)$.  Similarly, the remaining terms are $O(h^2)$, and the result follows.  
\end{proof}

Now in the $\theta$ variables, $Jf_h\tilde{Jf_h}$ is an approximation of a multiple of the $\delta$ function. 
We then have the following lemma.

\begin{lemma}\label{UndoneIntegral}
Letting $f_h$ be as in Lemma \ref{SingularFunctional}, we have
\[
H(f_h,f_h) =  \sigma(x)J(1)(x,\theta_0)\tilde{J}(1)(x,\theta_0)+ o(h).
\]
\end{lemma}

\begin{proof}
From the previous lemma, we have 
\[
H(f_h,f_h) = \int_{{S^{n-1}}} \sigma Jf_h \tilde{J}f_h \, d\theta + O(h^2).
\]
Now 
\[
Jf_h(x,\theta) = \exp\left(-\int_0^{\tau_{-}(x,\theta)}\sigma(x - s\theta)ds\right)f_h(x - \tau_{-}(x,\theta)\theta, \theta).
\]
Since $f$ is actually independent of $x$, we obtain
\[
Jf_h(x,\theta) = \exp\left(-\int_0^{\tau_{-}(x,\theta)}\sigma(x - s\theta)ds\right)f_h(\theta).
\]
We can write 
\[
H(f_h,f_h) = h^{1-n}\int_{{S^{n-1}}} h^{n-1}\sigma Jf_h \tilde{J}f_h \, d\theta + O(h^2),
\]
so that the argument of the integral is $O(1)$ in the $L^{\infty}$ sense. Then since $f_h$ is supported in a small neighbourhood of $\theta_0$, for small $h$ it follows from the Lebesgue differentiation theorem that we can replace $h^{n-1}Jf_h\tilde{J}f_h$ by its value at $\theta_0$, up to a term of $o(h)$.
Therefore
\begin{eqnarray*}
H(f_h,f_h)(x) &=& h^{1-n}\int_{\mathrm{supp}f_h} (h^{n-1}\sigma(x) Jf_h(x,\theta_0) \tilde{J}f_h(x,\theta_0)  + o(h))\, d\theta + O(h^2)\\
              &=& h^{1-n} (h^{n-1}\sigma(x) Jf_h(x,\theta_0) \tilde{J}f_h(x,\theta_0)  + o(h))\int_{\mathrm{supp}f_h} 1\, d\theta + O(h^2),\\
              &=& h^{n-1}\sigma(x)Jf_h(x,\theta_0)\tilde{J}f_h(x,\theta_0)+ o(h). \\
\end{eqnarray*}
Now
\begin{eqnarray*}
Jf_h(x,\theta_0) &=& \exp\left(-\int_0^{\tau_{-}(x,\theta_0)}\sigma(x - s\theta_0)ds\right)f_h(\theta_0). \\
                 &=& J(1)(x,\theta_0)f_h(\theta_0), \\
								 &=& J(1)(x,\theta_0)h^{\half(1-n)}. \\
\end{eqnarray*}
and similarly
\[
\tilde{J}f_h(x,\theta_0) = \tilde{J}(1)(x,\theta_0)h^{\half(1-n)}. 
\]
Therefore
\[
H(f_h,f_h) =  \sigma(x)J(1)(x,\theta_0)\tilde{J}(1)(x,\theta_0)+ o(h)
\]
as desired. Note in particular that the scaling on $f_h$ has been chosen precisely so $H(f_h,f_h)$ is $O(1)$ in $h$.
\end{proof}

Now the proof of Theorem \ref{AbsorptionRecovery} only requires one extra step.  

\begin{proof}[Proof of Theorem \ref{AbsorptionRecovery}]
From the previous lemma, we know that 
\begin{equation}\label{SigmaH}
H(f_h,f_h)(x) =  \sigma(x)J(1)(x,\theta_0)\tilde{J}(1)(x,\theta_0)+ o(h).
\end{equation}
Note that 
\begin{eqnarray*}
& & \theta_0 \cdot \grad (J(1)(x,\theta_0)\tilde{J}(1)(x,\theta_0)) \\
&=& (\theta_0 \cdot \grad J(1))(x,\theta_0)\tilde{J}(1)(x,\theta_0) + J(1)(x,\theta_0) (\theta_0 \cdot \grad\tilde{J}(1))(x,\theta_0)\\
&=& -\sigma J(1)(x,\theta_0)\tilde{J}(1)(x,\theta_0) + \sigma J(1)(x,\theta_0)\tilde{J}(1)(x,\theta_0)\\
&=& 0.
\end{eqnarray*}
Therefore the expression 
\[
J(1)(x,\theta_0)\tilde{J}(1)(x,\theta_0)
\] 
is constant along lines parallel to $\theta_0$. Another way to express this is to say that the quantity 
\[
J(1)(x + t\theta_0,\theta_0)\tilde{J}(1)(x+ t\theta_0,\theta_0)
\]
is independent of $t$, as long as $x+ t\theta_0$ lies in $X$.  Then if we pick $t = \tau_+ = \tau_{+}(x,\theta_0)$, so that $x+t\theta_0$ lies in $\Gamma_+$, then we know by definition of $\tilde{J}$ that
\[
\tilde{J}(1)(x +\tau_{+}\theta_0,\theta_0) = 1,
\]
so
\[
J(1)(x,\theta_0)\tilde{J}(1)(x,\theta_0)= J(1)(x + \tau_{+}\theta_0,\theta_0).
\]
Now we claim that
\begin{equation}
\label{AlbedoFact}
J(1)(x +\tau_{+}\theta_0,\theta_0) = h^{\half(n-1)}(\A_0(f_h)(x + \tau_{+}\theta_0,\theta_0) + O(h)),
\end{equation}
so that up to order $h$, we can determine $J(1)(x + \tau_{+}\theta_0,\theta_0)$ from the boundary data.  To prove this claim, recall from the definition of $J$ that 
\[
J(1)(x +\tau_{+}\theta_0,\theta_0) = \exp\left(-\int_{-\tau_{+}}^{\tau_{-}(x,\theta_0)}\sigma(x - s\theta_0)ds\right)
\]
Then $J(1)(x +\tau_{+}\theta_0,\theta_0)$ can be rewritten as 
\begin{eqnarray*}
& & \exp\left(-\int_{-\tau_{+}}^{\tau_{-}(x,\theta_0)}\sigma(x - s\theta_0)ds\right)h^{\half(1-n)}h^{\half(n-1)} \\
&=& \exp\left(-\int_{-\tau_{+}}^{\tau_{-}(x,\theta_0)}\sigma(x - s\theta_0)ds\right)f_h(\theta_0)h^{\half(n-1)} \\
&=& h^{\half(n-1)}Jf_h(x +\tau_{+}\theta_0,\theta_0) \\
&=& h^{\half(n-1)}(u(x +\tau_{+}\theta_0,\theta_0) + O(h)). \\
\end{eqnarray*}
Since $(x +\tau_{+}\theta_0,\theta_0) \in \Gamma_{+}$, 
\[
\A_0(f_h)(x + \tau_{+}\theta_0,\theta_0)=u(x +\tau_{+}\theta_0,\theta_0)
\]
by definition of $\A_0$, and this proves \eqref{AlbedoFact}. Returning now to \eqref{SigmaH}, we have
\[
H(f_h,f_h)(x) =  \sigma(x)h^{\half(n-1)}\A_0(f_h)(x + \tau_{+}\theta_0,\theta_0)+ o(h).
\]
Rearranging, we have
\[
\sigma(x) = \frac{H(f_h,f_h)(x)}{h^{\half(n-1)}\A_0(f_h)(x + \tau_{+}\theta_0,\theta_0)} + o(h),
\]
which is just equation \eqref{SigmaFromH}.  Note that Lemma \ref{UndoneIntegral} and the claim in equation \eqref{AlbedoFact} show that the numerator and denominator, respectively, of the fraction in \eqref{SigmaFromH} are both $O(1)$.

\end{proof}

\section{Recovering the Scattering Kernel}

Now we turn to the proof of Theorem \ref{ScatteringRecovery}.  We begin by defining the boundary sources $g^{\theta_1}_h$.  To do this, pick $\theta_1 \in S^{n-1}$ and let $h > 0$. {Define $f^{\theta_1}_h$ in the same manner as in the previous section, with $\theta_0$ replaced by $\theta_1$.} That is, put $f_h^{\theta_1}\in L^{\infty}({S^{n-1}})$ by
\[
f^{\theta_1}_h(\theta) = \left\{ \begin{array}{ll} h^{\half(1-n)} & \mbox{ if } |\theta - \theta_1| < h , \\ 
                                       0              & \mbox{ otherwise} .  \end{array} \right.
\]
Now define the function $s: \R \rightarrow \R$ by
\[
s(t) = \left\{ \begin{array}{ll} 1 &\mbox{ if } \lfloor t \rfloor \mbox{ is even},  \\
                                -1 &\mbox{ if } \lfloor t \rfloor \mbox{ is odd} .
                                 \end{array} \right.
\]
We can choose coordinates $x_1, \ldots, x_n$ on $\R^n$ so that when $S^{n-1}$ is embedded in $\R^n$, $\theta_1$ lies on the $x_n$ axis.  Then let
\[
g^{\theta_1}_h(\theta, x) = h^{\half(1-n)}f^{\theta_1}_h(\theta)s(x_1/h) \cdots s(x_{n-1}/h).
\]
Note that $g^{\theta_1}_h(\theta, x)$ is supported only for $\theta$ near $\theta_1$, and if we fix a $\theta$ near $\theta_1$, then $g^{\theta_1}_h(\theta, x) = \pm h^{1-n}$ is highly oscillatory as a function of $x$ in all directions perpendicular to $\theta_1$.  

Now let $u_{\theta_1}$ be the solution to \eqref{RTE} with boundary condition $u|_{\Gamma_{-}} = g^{\theta_1}_h|_{\Gamma_{-}}$. By Proposition \ref{UDecomp}, we have
\[
u_{\theta_1} = Jg^{\theta_1}_h + \sum_{j = 1}^{\infty} K^j (Jg^{\theta_1}_h).
\]
Since $\|g^{\theta_1}_h\|_{L^1} = O(1)$, Corollary \ref{PreservingSingularity} no longer guarantees us that $u_{\theta_1} - Jg^{\theta_1}_h$ is $O(h)$.  On the other hand, we can use the spatial oscillation of $g^{\theta_1}_h$ to prove the following lemma.

\begin{lemma}\label{uMinusJg}
Let $h> 0$ and let $u_{\theta_1}$ and $g^{\theta_1}_h$ be defined as above.  Then
\begin{equation}\label{uMinusJg1}
\|u_{\theta_1} - Jg^{\theta_1}_h\|_{L^{\infty}(X \times {S^{n-1}})} = O(1)
\end{equation}
and for fixed $x$,
\begin{equation}\label{uMinusJg2}
\|u_{\theta_1}(x,\cdot) - Jg^{\theta_1}_h(x,\cdot)\|_{L^{1}({S^{n-1}})} = o(h).
\end{equation}
Moreover, if $W_h \subset {S^{n-1}}$ is the subset of ${S^{n-1}}$ defined by
\[
W_h = \{ \theta \in {S^{n-1}}: |\theta - \theta_1| > h^{\half} \}
\]
then for small $h$
\begin{equation}\label{uMinusJg3}
\|u_{\theta_1} - Jg^{\theta_1}_h\|_{L^{\infty}(X \times W_h)} = o(h).
\end{equation}
\end{lemma} 

\begin{proof}
By Proposition \ref{UDecomp}, we have
\[
u_{\theta_1} = Jg^{\theta_1}_h + \sum_{j = 1}^{\infty} K^j (Jg^{\theta_1}_h).
\]
First, we note that for any fixed $x$, $g^{\theta_1}_h(x, \theta)$ is $\pm h^{1-n}$ in a neighbourhood of measure $h^{n-1}$ and zero otherwise, so $\| g^{\theta_1}_h(x, \cdot)\|_{L^1({S^{n-1}})} = O(1)$.  Then using~\eqref{JfL1}, we see that
\[
\|Jg^{\theta_1}_h(x, \cdot)\|_{L^{1}({S^{n-1}})} = O(1),
\]
and thus equation \eqref{A2wLinfty} implies that 
\begin{equation}\label{A2Jg}
\|A_2 Jg^{\theta_1}_h(x, \cdot)\|_{L^{\infty}({S^{n-1}})} = O(1).
\end{equation}
Eq.~\eqref{uMinusJg1} follows immediately. So far, we have not taken advantage of the spatial oscillation in $g^{\theta_1}_h$.  Notice that since $J$ is a multiplicative operator, $A_2$ is local in the spatial variables, and since both are positive operators, they preserve the spatial oscillation. In other words, we can still write
\[
A_2 Jg^{\theta_1}_h(x, \theta) = \alpha(x,\theta)s(x_1/h) \cdots s(x_{n-1}/h)
\]
where for fixed $\theta$, $\alpha(x,\theta)$, as a function of $x$, is independent of $h$. Now, recall that 
\[
(T_1^{-1}w)(x,\theta) = \int_0^{\tau_{-}(x,\theta)}\exp\left(-\int_0^{t}\sigma(x - s\theta)ds\right)w(x - t\theta, \theta)dt.
\]
Thus for $\theta$ outside of an $O(h^{\half})$ distance from $\theta_1$, the spatial oscillation of $A_2Jg^{\theta_1}_h$, combined with the Riemann-Lebesgue lemma, guarantees that 
\[
KJg^{\theta_1}_h(x, \theta) = T^{-1}_1 A_2 Jg^{\theta_1}_h(x, \theta) = o(h).
\]
In other words, 
\begin{equation}\label{uMinusJg4}
\|KJg^{\theta_1}_h\|_{L^{\infty}(X \times W_h)} = o(h).
\end{equation}
Now in ${S^{n-1}} \setminus W_h$, which has $O(h^{\half(n-1)})$ volume, we get from the estimates on $A_2 Jg^{\theta_1}_h(x, \cdot)$ that 
\[
\|T^{-1}_1 A_2 Jg^{\theta_1}_h(x, \cdot)\|_{L^{\infty}({S^{n-1}} \setminus W_h)} = O(1).
\]
Combining the two previous statements, we obtain
\begin{equation}\label{uMinusJg5}
\|K Jg^{\theta_1}_h(x, \cdot)\|_{L^1({S^{n-1}})} = \|T^{-1}_1 A_2 Jg^{\theta_1}_h(x, \cdot)\|_{L^1({S^{n-1}})} = o(h).
\end{equation}
Then, \eqref{A2wLinfty} says that 
\[
\|A_2 K Jg^{\theta_1}_h\|_{L^{\infty}(X \times {S^{n-1}})} = o(h).
\]
Therefore
\[
\|K^2 Jg^{\theta_1}_h\|_{L^{\infty}(X \times {S^{n-1}})} = \|T_1^{-1}A_2 K Jg^{\theta_1}_h\|_{L^{\infty}(X \times {S^{n-1}})} = o(h),
\]
and using the $L^{\infty}$ bounds on $K$, 
\begin{equation}\label{uMinusJg6}
\left\| \sum_{j=2}^{\infty} K^j Jg^{\theta_1}_h \right\|_{L^{\infty}(X \times {S^{n-1}})} = o(h).
\end{equation}
Combining \eqref{uMinusJg5} and \eqref{uMinusJg6} now gives \eqref{uMinusJg2}, and combining \eqref{uMinusJg4} and \eqref{uMinusJg6} gives \eqref{uMinusJg3}, completing the proof.

\end{proof}

We can now use Lemma~\ref{uMinusJg} to decompose the functional $H(g^{\theta_1}_h, g^{\theta_2}_h)$.

\begin{lemma}\label{ScatterHDecomp}
Let $\theta_1, \theta_2 \in {S^{n-1}}$, with $\theta_1 \neq \theta_2$.  Then for $h > 0$
\[
H(g^{\theta_1}_h, g^{\theta_2}_h) = \int_{S^{n-1}} A_2 Jg^{\theta_1}_h \tilde{J}g^{\theta_2}_h d\theta + o(h).
\] 
\end{lemma}

\begin{proof}
We can expand $H(g^{\theta_1}_h, g^{\theta_2}_h)$ as in Lemma~\ref{SingularFunctional} to get
\begin{eqnarray*}
H(g^{\theta_1}_h,g^{\theta_2}_h) &=& \int_{{S^{n-1}}}\sigma Jg^{\theta_1}_h\tilde{J}g^{\theta_2}_h \, d\theta +\int_{S^{n-1}}\sigma Jg^{\theta_1}_h (v-\tilde{J}g^{\theta_2}_h)d\theta \\
	   & & +\int_{S^{n-1}}\sigma (u-Jg^{\theta_1}_h)\tilde{J}g^{\theta_2}_h d\theta +\int_{S^{n-1}}\sigma(u-Jg^{\theta_1}_h) (v-\tilde{J}g^{\theta_2}_h)d\theta \\
	         & & +\int_{S^{n-1}} A_2(Jg^{\theta_1}_h)\tilde{J}g^{\theta_2}_h \, d\theta +\int_{S^{n-1}} A_2(Jg^{\theta_1}_h)(v-\tilde{J}g^{\theta_2}_h)d\theta \\
	         & & +\int_{S^{n-1}} A_2(u-Jg^{\theta_1}_h)\tilde{J}g^{\theta_2}_h d\theta+\int_{S^{n-1}} A_2(u-Jg^{\theta_1}_h)(v-\tilde{J}g^{\theta_2}_h)d\theta. \\
\end{eqnarray*}
Using \eqref{uMinusJg1} and \eqref{uMinusJg2} from Lemma \ref{uMinusJg}, we see that 
\[
\int_{S^{n-1}}\sigma(u-Jg^{\theta_1}_h) (v-\tilde{J}g^{\theta_2}_h)d\theta = o(h). 
\]
Moreover, for any fixed $x \in X$
\[
\|A_2(Jg^{\theta_1}_h)(x,\cdot)\|_{L^{\infty}({S^{n-1}})} \lesssim \|Jg^{\theta_1}_h(x,\cdot)\|_{L^1({S^{n-1}})} =  O(1),
\]
so combining this with \eqref{uMinusJg2}, we find that 
\[
\int_{S^{n-1}} A_2(Jg^{\theta_1}_h)(v-\tilde{J}g^{\theta_2}_h)d\theta = o(h).
\] 
Similar reasoning says that 
\[
\int_{S^{n-1}} A_2(u-Jg^{\theta_1}_h)\tilde{J}g^{\theta_2}_h d\theta+\int_{S^{n-1}} A_2(u-Jg^{\theta_1}_h)(v-\tilde{J}g^{\theta_2}_h)d\theta = o(h).
\]
Therefore we have
\begin{eqnarray*}
H(g^{\theta_1}_h,g^{\theta_2}_h) &=& \int_{{S^{n-1}}}\sigma Jg^{\theta_1}_h\tilde{J}g^{\theta_2}_h \, d\theta +\int_{S^{n-1}}\sigma Jg^{\theta_1}_h (v-\tilde{J}g^{\theta_2}_h)d\theta \\
	        & & +\int_{S^{n-1}}\sigma (u-Jg^{\theta_1}_h)\tilde{J}g^{\theta_2}_h d\theta +\int_{S^{n-1}} A_2(Jg^{\theta_1}_h)\tilde{J}g^{\theta_2}_h \, d\theta \\
	        & & +o(h) \\
\end{eqnarray*}
Now examine the term
\[
\int_{S^{n-1}}\sigma (u-Jg^{\theta_1}_h)\tilde{J}g^{\theta_2}_h d\theta .
\]
For small enough $h$, the function $\tilde{J}g^{\theta_2}_h(x, \theta)$, as a function of $\theta$, has support only on the set 
\[
W_h = \{ \theta \in {S^{n-1}}: |\theta - \theta_1| > h^{\half}\}.
\]
Then \eqref{uMinusJg3} from Lemma \ref{uMinusJg} says that 
\begin{eqnarray*}
\left|\int_{S^{n-1}}\sigma (u-Jg^{\theta_1}_h)\tilde{J}g^{\theta_2}_h d\theta \right| &\lesssim& \|u-Jg^{\theta_1}_h\|_{L^{\infty}(W_h)}\|\tilde{J}g^{\theta_2}_h\|_{L^1({S^{n-1}})} \\
&=& o(h). \\
\end{eqnarray*}
Similarly
\[
\int_{S^{n-1}}\sigma Jg^{\theta_1}_h (v-\tilde{J}g^{\theta_2}_h)d\theta = o(h),
\]
so
\[
H(g^{\theta_1}_h, g^{\theta_2}_h) = \int_{S^{n-1}} \sigma Jg^{\theta_1}_h \tilde{J}g^{\theta_2}_h d\theta + \int_{S^{n-1}} A_2 Jg^{\theta_1}_h \tilde{J}g^{\theta_2}_h d\theta + o(h).
\] 
Finally, if $h$ is sufficiently small compared to $|\theta_1 - \theta_2|$, then $Jg^{\theta_1}$ and $Jg^{\theta_2}$ have disjoint supports as functions of $\theta$. Therefore the first integral on the right side vanishes, and
\[
H(g^{\theta_1}_h, g^{\theta_2}_h) = \int_{S^{n-1}} A_2 Jg^{\theta_1}_h \tilde{J}g^{\theta_2}_h d\theta + o(h)
\] 
as desired.
\end{proof}

We are now ready for the proof of Theorem \ref{ScatteringRecovery}.  

\begin{proof}[Proof of Theorem \ref{ScatteringRecovery}]  

From Lemma \ref{ScatterHDecomp}, we have
\begin{equation}\label{ScatteringQuantity1}
H(g^{\theta_1}_h, g^{\theta_2}_h) = \int_{S^{n-1}} A_2 Jg^{\theta_1}_h \tilde{J}g^{\theta_2}_h d\theta + o(h).
\end{equation}
Now
\[
\tilde{J}g^{\theta_2}_h(x,\theta) = \exp\left(-\int_0^{\tau_{+}(x,\theta)}\sigma(x+s\theta)ds\right)g^{\theta_2}_h(x+\tau_{+}(x,\theta)\theta, \theta) .
\]
For fixed $x \in X$, $g^{\theta_2}_h(x,\theta) = \pm h^{(1-n)/2}f^{\theta_2}_h(\theta)$, where the sign depends on $x$. Therefore
\[
\tilde{J}g^{\theta_2}_h(x,\theta) = \pm h^{(1-n)/2}\exp\left(-\int_0^{\tau_{+}(x,\theta)}\sigma(x + s\theta)ds\right)f^{\theta_2}_h(\theta) .
\]
Since $f^{\theta_2}_h(\theta)$ is supported only in a small neighborhood of $\theta_2$, we can substitute the above into \eqref{ScatteringQuantity1} and use the Lebesgue differentiation theorem to get for small $h$ that
\begin{equation}
\label{UndoneScatteringIntegral}
H(g^{\theta_1}_h, g^{\theta_2}_h) = \pm A_2(Jg^{\theta_1}_h)(x,\theta_2)\exp\left(-\int_0^{\tau_{+}(x,\theta_2)}\sigma(x + s\theta_2)ds\right) + o(h).
\end{equation}
Since $\sigma$ is known, we obtain
\begin{equation}\label{ScatteringQuantity2}
 H(g^{\theta_1}_h, g^{\theta_2}_h)\exp\left(\int_0^{\tau_{+}(x,\theta_2)}\sigma(x + s\theta_2)ds\right)  = \pm A_2(Jg^{\theta_1}_h)(x,\theta_2) + o(h).
\end{equation}
Writing out the operator $A_2$ in full, we can rewrite the above as
\[
A_2(Jg^{\theta_1}_h)(x,\theta_2) = \int_{S^{n-1}} k(x, \theta_2, \theta ') Jg^{\theta_1}_h(x,\theta ') d\theta '.
\]
Now
\[
Jg^{\theta_1}_h(x,\theta) = \pm h^{(1-n)/2}\exp\left(-\int_0^{\tau_{-}(x,\theta)}\sigma(x - s\theta)ds\right)f^{\theta_1}_h(\theta).
\]
Therefore we can repeat the argument used to obtain \eqref{UndoneScatteringIntegral} to get
\[
A_2(Jg^{\theta_1}_h)(x,\theta_2) = \pm k(x,\theta_2, \theta_1)\exp\left(-\int_0^{\tau_{-}(x,\theta_1)}\sigma(x - s\theta_1)ds\right) + o(h). 
\]
Therefore \eqref{ScatteringQuantity2} can be rewritten as
\[
H(g^{\theta_1}_h, g^{\theta_2}_h)\exp\left(\int_0^{\tau_{+}(x,\theta_2)}\sigma(x + s\theta_2)ds\right) = \pm k(x,\theta_2, \theta_1)\exp\left(-\int_0^{\tau_{-}(x,\theta_1)}\sigma(x - s\theta_1)ds\right) + o(h).
\]
Rearranging, we have 
\[
k(x,\theta_2, \theta_1) = \left| H(g^{\theta_1}_h, g^{\theta_2}_h)\exp\left(\int_0^{\tau_{+}(x,\theta_2)}\sigma(x + s\theta_2)ds+\int_0^{\tau_{-}(x,\theta_1)}\sigma(x - s\theta_1)ds\right)\right| +o(h).
\]
This proves \eqref{KFromH}. Repeating for all $\theta_1, \theta_2$ pairs gives $k$.

\end{proof}

\section*{Acknowledgements}

The authors were supported in part by the NSF grants DMS-1619907 and DMR-1120923 to JCS.

\end{document}